\documentclass[12pt,reqno]{article}

\usepackage[usenames]{color}
\usepackage{amssymb}
\usepackage{amsmath}
\usepackage{amsthm}
\usepackage{amsfonts}
\usepackage{amscd}
\usepackage{graphicx}

\usepackage[colorlinks=true,
linkcolor=webgreen,
filecolor=webbrown,
citecolor=webgreen]{hyperref}

\definecolor{webgreen}{rgb}{0,.5,0}
\definecolor{webbrown}{rgb}{.6,0,0}

\usepackage{color}
\usepackage{fullpage}
\usepackage{float}

\usepackage{graphics}
\usepackage{latexsym}
\usepackage{epsf}
\usepackage{breakurl}

\newcommand{\seqnum}[1]{\href{https://oeis.org/#1}{\rm \underline{#1}}}

\DeclareMathOperator{\per}{\textsc{Per}}
\DeclareMathOperator{\pal}{\textsc{Pal}}
\DeclareMathOperator{\paltwo}{\textsc{Paltwo}}
\DeclareMathOperator{\faceq}{\textsc{Faceq}}
\DeclareMathOperator{\novel}{\textsc{Novel}}
\DeclareMathOperator{\ispp}{\textsc{Ispp}}
\DeclareMathOperator{\countpp}{\textsc{Countpp}}
\DeclareMathOperator{\BWT}{BWT}
\DeclareMathOperator{\pp}{pp}
\def\andd{\, \wedge \, }

\newenvironment{smallarray}[1]
{\null\,\vcenter\bgroup\scriptsize
\arraycolsep=.13885em
\hbox\bgroup$\array{@{}#1@{}}}
{\endarray$\egroup\egroup\,\null}

\def\Sigmastar{\Sigma^*}

\begin{document}

\theoremstyle{plain}
\newtheorem{theorem}{Theorem}
\newtheorem{corollary}[theorem]{Corollary}
\newtheorem{lemma}[theorem]{Lemma}
\newtheorem{proposition}[theorem]{Proposition}

\theoremstyle{definition}
\newtheorem{definition}[theorem]{Definition}
\newtheorem{example}[theorem]{Example}
\newtheorem{conjecture}[theorem]{Conjecture}

\theoremstyle{remark}
\newtheorem{remark}[theorem]{Remark}

\title{Some Remarks on Palindromic Periodicities}

\author{Gabriele Fici\\
 Dipartimento di Matematica e Informatica\\ Universit\`a degli Studi di Palermo\\ 90133 Palermo \\
 Italy\\
\href{mailto:gabriele.fici@unipa.it}{\tt gabriele.fici@unipa.it}\\
\and
Jeffrey Shallit\\
School of Computer Science\\
University of Waterloo\\
Waterloo, ON  N2L 3G1\\
Canada\\
\href{mailto:shallit@uwaterloo.ca}{\tt shallit@uwaterloo.ca}\\
\and Jamie Simpson\\
130 Preston Point Rd, East \\
Fremantle, WA 6158\\
Australia\\
\href{mailto:jamiesimpson320@gmail.com}{\tt jamiesimpson320@gmail.com}}

\maketitle

\begin{abstract}
We say a finite word $x$ is a {\it palindromic periodicity} 
if there exist two palindromes $p$ and $s$ such
that $|x| \geq |ps|$ and
$x$ is a prefix of the word $(ps)^\omega = pspsps\cdots$.  In this paper we examine the palindromic periodicities occurring in some classical infinite words, such as Sturmian words, episturmian words, the Thue--Morse word, the period-doubling word, the Rudin--Shapiro word, the paperfolding word, and the Tribonacci word, and prove a number of results about them.   We also prove results about words with the smallest number of palindromic periodicities.
\end{abstract}

\section{Introduction}

Recently, the third author introduced the notion of {\it palindromic periodicity}
\cite{Simpson:2024}.   A finite word $x$ is called  a palindromic periodicity
if there exist two palindromes $p$ and $s$ such
that $|x| \geq |ps|$ and
$x$ is a prefix of the word $(ps)^\omega = pspsps\cdots$.
He proved a number of interesting results about these words.
As an example, $x = 121344312134$ is a palindromic periodicity, as can be
seen by taking $p = 121$, $s =3443$.   Note that in this example, $x$ itself is not the product
of two palindromes.  

Products of two palindromes, also called \emph{symmetric words} were studied, for example, by Kemp \cite{Kemp:1982}, Brlek et al.~\cite{Brlek&Hamel&Nivat&Reutenauer:2004}, Guo et al.~\cite{Guo&Shallit&Shur:2015}, and Borchert and Rampersad \cite{Borchert&Rampersad:2015}. Particular cases of symmetric words are the words that have a perfectly clustered Burrows--Wheeler transform. Recall that the \emph{Burrows--Wheeler Transform} (BWT) of a word $w$ is the word obtained by concatenating the last letters of the conjugates of $w$ sorted in lexicographic order. For example, if $w=0120$, the list of sorted conjugates of $w$ is $\{0012,0120,1200,2001\}$, so the BWT of $w$ is $2001$. By definition, the BWT of $w$ is the same as the BWT of any conjugate of $w$.
A word over the ordered alphabet $\Sigma=\{0,1,\ldots ,k-1\}$ is said to have \textit{perfectly clustered BWT} if $\BWT(w)$ belongs to $(k-1)^*\cdots 1^*0^*$.
It is known that a binary word has perfectly clustered BWT if and only if it is a (power of) a standard Sturmian word~\cite{Mantaci&Restivo&Sciortino:2003}.
Over the ternary alphabet $\{0,1,2\}$, words with perfectly clustered BWT are related to the three-distance theorem~\cite{Berthe&Reutenauer:2024}.
In \cite{Simpson&Puglisi:2008}, Simon Puglisi and the third author proved that if a ternary word has perfectly clustered BWT, then it is symmetric. This result was extended to alphabets of any size by Restivo and Rosone~\cite{Rosone&Restivo:2009}. In particular, then, every word that has a perfectly clustered BWT is a palindromic periodicity. 


\bigskip 

In this note, we prove a number of new results about the concept of palindromic periodicity.  For example, we prove that every factor of a Sturmian word is a palindromic periodicity. This result can be extended to the larger class of trapezoidal (or stiff) words, that are finite words with at most $n+1$ distinct factors of length $n$ for any $n$, and to all prefixes (but not all factors) of standard episturmian words.

Using the free software tool {\tt Walnut}
\cite{Mousavi:2016,Shallit:2023},
we examine the palindromic periodicities occurring in
certain classic infinite words, such as the Thue--Morse word, Rudin--Shapiro
word, period-doubling word, paperfolding word, Fibonacci word, and Tribonacci
word.

\section{General remarks}

We say that a nonempty word $z$ is a {\it word-period\/} of a word $w$ if $w$ is a prefix of $z^\omega =zzz\cdots$. The shortest word-period of a word is called its {\it fractional root}.   For example,
{\tt ent} is a word-period of
the French word {\tt entente} and also its fractional root. The length of a word-period of a word $w$ is called {\it a period\/} of $w$. We call the length of the fractional root of $w$ {\it the period\/} of $w$.

A word $v$ is a  {\it border\/} of a word $w$ if $v$ is both a prefix and a suffix of $w$. A word $w$ is unbordered if it does not have non-trivial borders (i.e., its only borders are the empty word $\varepsilon$ and the word $w$ itself). A word is unbordered if and only if it coincides with its fractional root.

The reverse of a finite word $w$ is denoted $w^R$.  If $w = w^R$, then $w$ is a {\it palindrome}.

Two words are {\it conjugate} if one is a cyclic shift of the other.
A word is called \textit{symmetric\/} if it is the concatenation of two palindromes; or, equivalently, if it is a conjugate of its reverse~\cite{Brlek&Hamel&Nivat&Reutenauer:2004}. A palindrome is a particular case of a symmetric word, in which one of the two words is the empty word.

Note that every symmetric word is a palindromic periodicity. Actually, it follows from the definition that a word $w$ is a palindromic periodicity if and only if $w$ has a symmetric word-period. In particular, if the fractional root of $w$ is symmetric, then $w$ is a  palindromic periodicity. However, $0100110$ is a palindromic periodicity (since it is symmetric), yet its fractional root $010011$ is not symmetric, so the converse does not hold in general.

\begin{remark}
    A word $w$ is symmetric if and only if $w^n$ is symmetric for any $n>1$. Indeed, if $w=uv$, with $u,v$ palindromes, then $w^n=(uv)^{n-1}uv$, and $(uv)^{n-1}u$ is a palindrome.

    An analogous property does not hold for palindromic periodicities. For example, $010110$ is a palindromic periodicity, but its square $010110010110$ is not.
\end{remark}

By the previous remark, we can suppose that the symmetric word-period of a nonempty palindromic periodicity is primitive. Since a word $w$ cannot have two distinct primitive word-periods both shorter than half the length of $w$ (as a consequence of  the Fine and Wilf theorem \cite{Fine&Wilf:1965}), it follows that if the fractional root of a palindromic periodicity $w$ is not symmetric, then $w$ has a symmetric word-period that is longer than half the length of $w$.

\section{Results for certain classes of words}

In this section, we prove that for certain classes of words the fractional root is always symmetric. As we saw in the previous section, this implies in particular that these words are palindromic periodicities.

A contiguous block in a word $w$ is called a {\it factor\/} of $w$.
Recall that an infinite binary word is called {\it Sturmian\/}
if it has exactly $n+1$ distinct factors of length $n$, for all $n \geq 0$.  A finite binary word is called Sturmian if it is a factor of an infinite Sturmian word.

A finite Sturmian word $u$ is a \textit{central word} if $0u0$, $0u1$, $1u0$ and $1u1$ are all Sturmian words. Central words are palindromes.

A finite Sturmian word is \textit{standard} if it is of the form $u01$ or $u10$, where $u$ is a central word. In particular, the words  $u01$ and $u10$ are conjugate~\cite{Luca&Mignosi:1994} and therefore a standard word is a conjugate of its reversal, i.e., it is symmetric.

We prove the following result, originally conjectured by the second author.
\begin{theorem}\label{thm:sturm}
Every nonempty factor of a Sturmian word is a palindromic periodicity.
\end{theorem}

\begin{proof}
A finite word is Sturmian if and only if its fractional root is a conjugate of a standard Sturmian word \cite{Luca&DeLuca:2006a} and every standard Sturmian word is a conjugate of its reversal~\cite{Luca&Mignosi:1994}. Therefore, every finite Sturmian word has a symmetric fractional root, and hence it is a palindromic periodicity.
\end{proof}

Notice that not every palindromic periodicity is a factor of a Sturmian word, with $0011$ being a smallest example.

A finite word is \textit{trapezoidal}~\cite{Luca:1999} (or \textit{stiff}~\cite{Heinis:2001}) if it has at most $n+1$ distinct factors of length $n$ for every $n$. Notice that trapezoidal words are binary words by definition. Finite Sturmian words are trapezoidal, but there are trapezoidal words that are not Sturmian (e.g., $0011$).  

Recall that a binary word $w$ is not Sturmian if and only if there exists a word $u$ such that both $0u0$ and $1u1$ are factors of $w$~\cite{Coven&Hedlund:1973}. The pair $(0u0,1u1)$ is called a \textit{pathological pair}. Moreover, it is possible to prove that the word $u$ in a pathological pair of minimal length is a central word~\cite{Bucci&DeLuca&Fici:2013}.
The structure of non-Sturmian trapezoidal words was described by D'Alessandro in~\cite{DAlessandro:2002}: Let $w$ be a binary word that is not Sturmian. Then $w$ is trapezoidal if and only if $w=pq$, where the fractional roots of $p^R$ and $q$ are $z_{0u0}$ and $z_{1u1}$. Moreover, $p$ and $q$ are Sturmian words.

\begin{example}
 The word $w=00001010$ is trapezoidal but not Sturmian. Its pathological pair of minimal length is $(000,101)$, so that $z_{000}=0$ and $z_{101}=10$. The word $w$ can be written as $w=pq$, with $p=0000$ and $q=1010$, where $z_{p^R}=z_{000}$ and $z_q=z_{101}$. 
\end{example}

\begin{theorem}\label{thm:trap}
  Every trapezoidal word has a symmetric fractional root; hence it is a palindromic periodicity.
\end{theorem}


Before proving Theorem~\ref{thm:trap}, we need to recall some preliminary results.


\begin{proposition}\cite{deLuca:1997,DAlessandro:2002}\label{prop:Pprime}
Let $u$ be a  word. Then $u$ is central if and only if $u$ is a power of a single letter or there exist central words $P$ and $Q$, with $P$ a proper prefix of $Q$, such that $u=PxyQ=QyxP$, where $\{x,y\}=\{0,1\}$.  

Moreover, $Q=(Pxy)^\ell P'=P' (yxP)^\ell$, for some $\ell>0$ and $P'$ is a palindrome such that either $P'=Px$ or $P'y$ is a prefix of $P$. 
\end{proposition}

\begin{proposition}\cite{DAlessandro:2002}\label{prop:radix}
    Let $u=PxyQ=QyxP$, with $|P|<|Q|$, be a central word that is not a power of a single letter. Then the fractional root of $xux$ is $xQy$, and the fractional root of $yuy$ is $yPx$.
\end{proposition}

\begin{proof}[Proof of Theorem~\ref{thm:trap}]


Let $w$ be a trapezoidal word. Since for Sturmian words the statement is true by Theorem~\ref{thm:sturm}, we can suppose that $w$ is not Sturmian.

By the characterization of non-Sturmian trapezoidal words, $w$ contains a pathological pair of minimal length $(0u0,1u1)$, where $u$ is a central word, and $w$ can be written as $w=pq$, where the fractional roots of $p^R$ and $q$ are $z_{0u0}$ and $z_{1u1}$.

If $u$ is a power of a letter, the statement is readily verified. So suppose that $u$ is not a power of a letter. By Proposition~\ref{prop:radix}, there exist central words $P$ and $Q$, with $P$ a proper prefix of $Q$, such that $u=PxyQ=QyxP$ and $z_{xux}=xQy$ and $z_{yuy}=yPx$, with $\{x,y\}=\{0,1\}$. 

So, either $w$ is a factor of a word in
\[(yQx)^* (yPx)^{*} \]
or a factor of a word in 
\[(xPy)^{*}  (xQy)^*.\]

Suppose first that $w$ is minimal, in the sense that it coincides with the concatenation of its pathological pair of minimal length, i.e., $\{p,q\}=\{xux,yuy\}$, so that $w$ is equal to  $w=xuxyuy$ or to $yuyxux$. Then $w=z_w$ since $w$ is unbordered, and in this case the claim is true. 

Suppose now to extend $w$ on the right into a word with a different root. That is, let $w'=pq'$ be the shortest word that has $w$ as a prefix and has root $z_{w'}\neq z_w$. For every proper prefix of $w'$ the claim is verified since the root has not changed. So we have  $w'=z_{w'}$, and we claim that $w'=xuxyu'y$ or $w'=yuyxu'x$ for a central word $u'$, so that $z_{w'}$ is symmetric and the claim is proved for $w'$. 

Suppose first $w=yuyxux$, so that $w$ and $w'$ are words in $(xPy)^{*}  (xQy)^*$. In this case,  we have $w'=w\cdot yQx$. In fact, by the characterization of non-Sturmian trapezoidal words, $q'$ must have fractional root $xQy$, and indeed $q'=xux\cdot xQy=xQyxPx\cdot yQx= xQyxQyxPx$, and $Px$ is a prefix of $Q$ by Proposition~\ref{prop:Pprime}. Since $yQy$ is a prefix of $p=yQyxPy$, and hence of $w'$, we have that $w\cdot yQx$ is the shortest right extension of $w$ that has a fractional root different from that of $w$. Finally, $q'=xux\cdot yQx=xQyxPxyQx$ and the word $u'=QyxPxyQ$ is a central word by Proposition \ref{prop:Pprime} since it can be written as $uxyQ=Qyxu$. 

Let us now assume $w=xuxyuy$, so that $w$ and $w'$ are words in $(yQx)^{*}  (yPx)^*$. Then in this case,  we have $w'=w\cdot xPy$. In fact,  by the characterization of non-Sturmian trapezoidal words, $q'$ must have fractional root $yPx$, and indeed $q'=yuy\cdot xPy = yPxyQyxPy\cdot yQx= yPxyPxyQy = yPxyPxy(Pxy)^\ell P'y$, and $P'y$ is a prefix of $Pxy$ by Proposition~\ref{prop:Pprime}. Since $xPx$ is a prefix of $p=xPxyQx$, and hence of $w'$, we have that $w\cdot xPy$ is the shortest right extension of $w$ that has a fractional root different from that of $w$. Finally, $q'=yuy\cdot xPy= yPxyQyxPy$ and the word $u'=PxyQyxP$ is a central word by Proposition \ref{prop:Pprime} since it can be written as $Pxyu=uyxP$.

Let us now consider the extensions on the left of $w'$. Again, the extension is unique  by the characterization of non-Sturmian trapezoidal words. But in this case, for each letter we add to the left, the fractional root will change. In particular, if the length of the fractional root does not change, the new fractional root is just a rotation of the previous one, and the claim is therefore proved. So let $w''=p'q'$ be the shortest left extension of $w'$ (i.e., such that $p$ is a suffix of $p'$) that has a  fractional root longer than that of $w'$. We claim that $w''=xu''xyu'y$ for a central word $u''$. 

Suppose first $w''\in (yQx)^*(yPx)^*$. By the characterization of non-Sturmian trapezoidal words, $(p')^R$ must have fractional root $(xQy)^R=yQx$, and indeed
$p'=xQy\cdot p = xQyxux= xQy xPxyQx=xPxyQxyQx$ and $xP$ is a suffix of $Q$ by Proposition \ref{prop:Pprime}. Since $z_{w'}$ ends in $yQy$, we have that $xQyw'$ is the shortest left extension of $w'$ that has a different fractional root. Finally, $z_{w''}=w''=p'q'$ and $p'=xQy\cdot p = xQyxux = xQy xPxyQx=xu''x$ with $u''=Qyxu=uxyQ$, and $u''$ is a central word by Proposition \ref{prop:Pprime}. 

If instead $w''\in (xPy)^*(xQy)^*$, $p'$ must have fractional root $(xPy)^R=yPx$, and indeed $p'=yPx\cdot p = yPx yuy = yPx yPxyQy= yPxyPxy(Pxy)^\ell P'y$, and $P'y$ is a prefix of $Pxy$ by Proposition~\ref{prop:Pprime}. Since $z_{w'}$ ends in $yPy$, we have that $yPxw'$ is the shortest left extension of $w'$ that has a different fractional root. Finally, $z_{w''}=w''=p'q'$ and $p'=yPx\cdot p=yPx yQyxPy=yu''y$ with $u''=Pxyu=uyxP$, and $u''$ is a central word by Proposition \ref{prop:Pprime}. 

Since by D'Alessandro's characterization, every non-Sturmian trapezoidal word is obtained by extending periodically on the right and on the left a word of the form $xuxyuy$ or $yuyxux$, the proof is complete.
\end{proof}

\begin{example}
The word $w=010001000100010010010010010$ is a non-Sturmian trapezoidal word. Its fractional root is $z_w=010001000100010010010010$ (and we have $w=z_w010$).
The pathological pair of minimal length of $w$ is  $(0001000,1001001)=(0u0,1u1)$, where $u=00100$. The fractional root of $0u0$ is $z_{0u0}=0001$, while the fractional root of $1u1$ is $z_{1u1}=100$.

The word $z_w$ is a conjugate of $z'=100010001000100100100100=(1000)^3(100)^4=(z_{0u0}^R)^3(z_{1u1})^4$. The word $z' =1000100010001\cdot 00100100100$ is symmetric, and hence $z_w$ is symmetric too.
\end{example}



The classes of Sturmian, trapezoidal, having symmetric fractional root, and palindromic periodicity form a strict inclusion hierarchy within the set of all finite binary words, as illustrated in Figure~\ref{fig1}.
\begin{figure}[htb]
\begin{center}
    \includegraphics[width=5in]{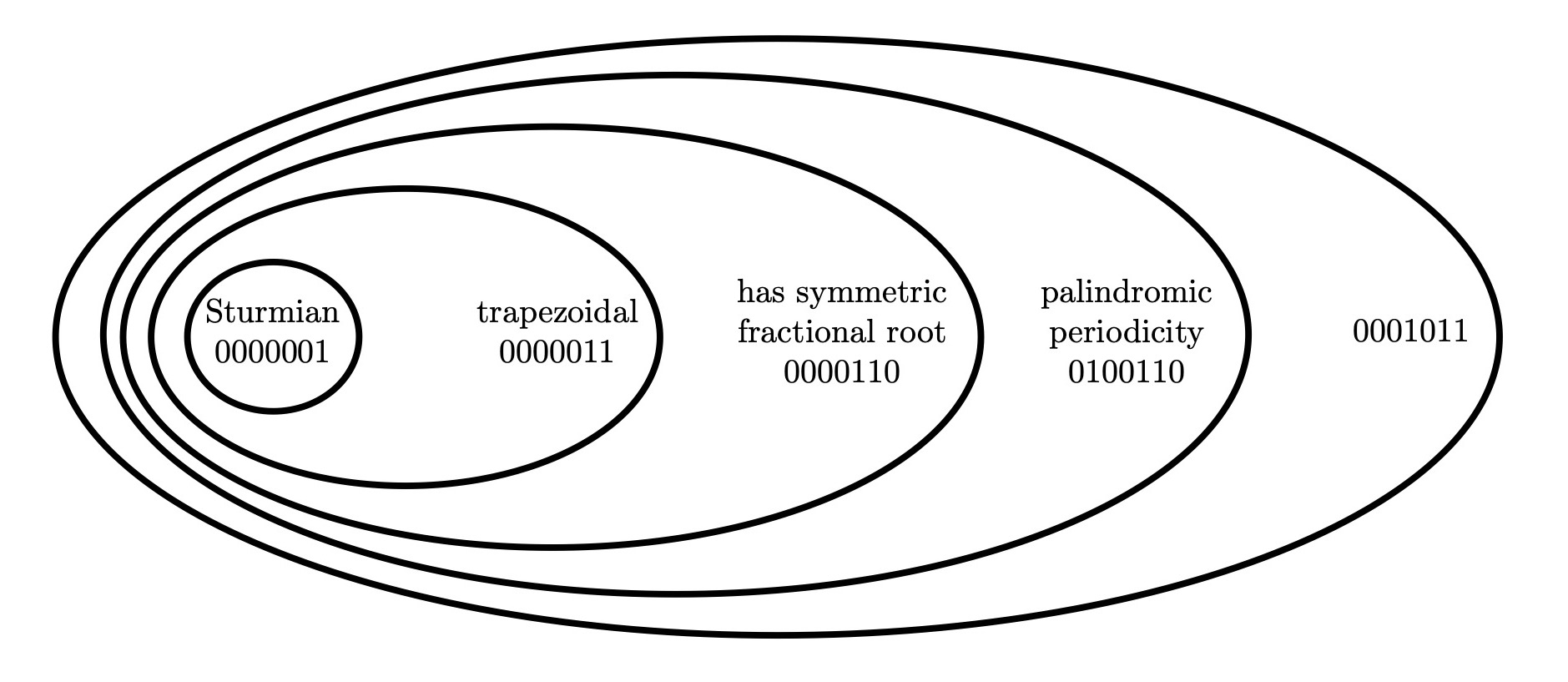}
\end{center}
\caption{The inclusion hierarchy.}
\label{fig1}
\end{figure}

Recall that a word of length $n$ contains at most $n$ nonempty distinct palindromic factors. A word of length $n$ is called \textit{rich} if it contains $n$ distinct palindromic factors \cite{Glen&Justin&Widmer&Zamboni:2009}.

Every trapezoidal word is rich. But there are rich words that are not palindromic periodicities, e.g., $001011$. On the other hand, there are palindromic periodicities that are not rich, a shortest example being $001001101011$. 
This is
illustrated in Figure~\ref{fig2}.
\begin{figure}[htb]
\begin{center}
    \includegraphics[width=4in]{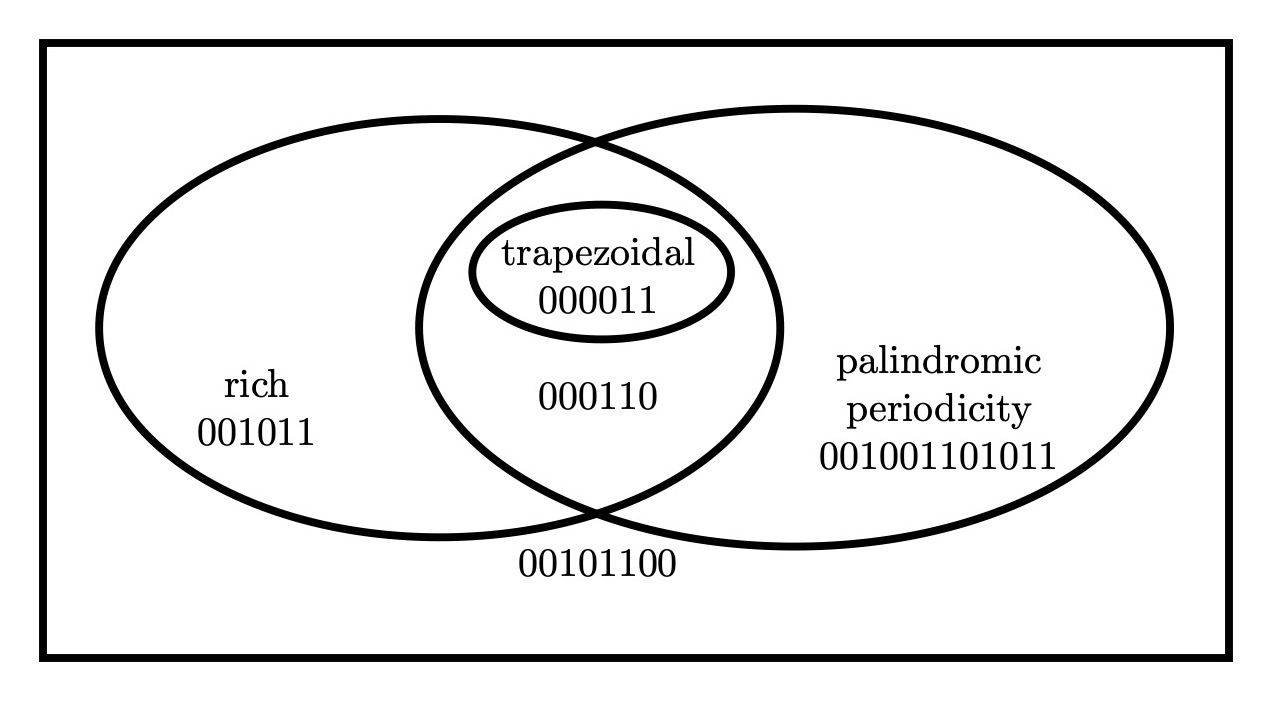}
\end{center}
\caption{Rich words, trapezoidal words, and palindromic periodicities.}
\label{fig2}
\end{figure}

However, if a word is rich and closed, then it is a palindromic periodicity. A word is called \textit{closed}~\cite{Fici:2011} (or \textit{periodic-like}~\cite{Carpi&deLuca:2001}) if it has length $1$ or it has a factor that  appears only as a prefix and as a suffix,  i.e., without internal occurrences. For example, $00$, $0110$, and $01010$ are closed, while $01$, $0010$, and $0101001$ are not.

\begin{proposition}\label{prop:richandclosed}
    If a word is rich and closed, then it is a  palindromic periodicity. 
\end{proposition}

\begin{proof}
   In~\cite{Bucci&Luca&DeLuca:2009}, Bucci, de Luca, and De Luca proved  that if a word is rich and closed, then its fractional root is symmetric. Hence it is a palindromic periodicity.
\end{proof}

Notice that the language of rich words is factorial (closed under taking factors) while the language of palindromic periodicities is not.

\bigskip

Let us now consider larger alphabets.

Recall that an infinite word is \textit{episturmian} if its set of finite factors is closed under reversal and it has at most one left special factor (i.e., a factor that occurs preceded by at least two different letters) of each length. Moreover, an episturmian word is \textit{standard} if all of its left special factors are prefixes of it. Episturmian words are a generalization of Sturmian words to larger alphabets. See~\cite{Glen&Justin:2009} for a survey.

\begin{theorem}
Every nonempty prefix of a standard episturmian word is a palindromic periodicity.
\end{theorem}

\begin{proof}
In~\cite{Luca&DeLuca:2006b}, de Luca and De Luca proved  that for every nonempty prefix $w$ of a standard episturmian word, the fractional root of $w$ is symmetric. Hence $w$ is a palindromic periodicity.
\end{proof}

In particular, then, all nonempty prefixes of the Tribonacci word, the fixed point of the morphism
$0 \rightarrow 01$, $1 \rightarrow 02$, $2 \rightarrow 0$, are palindromic periodicities. However, not all factors of the Tribonacci word are palindromic periodicities, $102$ being a smallest example.    See Section~\ref{trib-sec} for more discussion.

\section{Results for automatic words}

Let $\pp_{\bf x} (n)$ denote the number of length-$n$
factors of an infinite word $\bf x$ that are palindromic periodicities.

\begin{theorem}
Let $\bf x$ be a generalized automatic infinite word.   Then there exists
a finite automaton taking two inputs $i$ and $n$ in parallel, and accepting
if and only if ${\bf x}[i..i+n-1]$ is a palindromic periodicity.
\end{theorem}

\begin{proof}
The idea is to create a first-order logical formula that evaluates to
{\tt TRUE} if ${\bf x}[i..i+n-1]$ is a palindromic periodicity.  We
do this in several steps:
\begin{itemize}
\item $\per(i,n,p)$ asserts that ${\bf x}[i..i+n-1]$ has period $p$
\item $\pal(i,n)$ asserts that ${\bf x}[i..i+n-1]$ is a palindrome
\item $\paltwo(i,n)$ asserts that ${\bf x}[i..i+n-1]$ is the product of
two (possibly empty) palindromes 
\item $\ispp(i,n)$ asserts that
${\bf x}[i..i+n-1]$ is a palindromic periodicity.
\end{itemize}
These can be defined as follows:
\begin{align*}
\per(i,n,p) &:= \forall t\ (t\geq i \andd t<i+n-p) \implies {\bf x}[t]={\bf x}[t+p] \\
\pal(i,n) &= \forall t,u \ (t\geq i \andd t<i+n \andd t+u=2i+n-1) \implies
	{\bf x}[t] = {\bf x}[u] \\
\paltwo(i,n) &= \exists m \ m\leq n \andd \pal(i,m) \andd \pal(i+m,n-m) \\
 \ispp(i,n) &:= \exists p \ p\geq 1 \andd p\leq n \andd n\geq 1  \andd
	\per(i,n,p) \andd \paltwo(i,p) .
\end{align*}
Now, by the classic result of
Bruy\`ere et al.~\cite{Bruyere&Hansel&Michaux&Villemaire:1994}, we know there is an algorithm to convert these first-order statements to the desired automata.
\end{proof}

\begin{corollary}
Let $\bf x$ be a generalized automatic infinite word.  There is an algorithm that, given the automaton computing $\bf x$, decides whether $\bf x$ contains arbitrarily long palindromic periodicities.
\end{corollary}

\begin{proof}
It suffices to define the appropriate first-order formula, which is
$$ \forall m\ \exists i,n\ n>m \andd \ispp(i,n).$$
This can then be translated into an automaton of one state that either accepts everything (and so $\bf x$ contains arbitrarily long
palindromic periodicities) or nothing (and so $\bf x$ does not).
\end{proof}

\begin{corollary}
Let $\bf x$ be a generalized automatic infinite word.   Then the number $\pp_{\bf x} (n)$ of length-$n$ factors that are palindromic periodicities is a regular sequence, and furthermore there is an algorithm to compute a linear representation for it.
\end{corollary}

\begin{proof}
We need three additional automata, computable from 
the following first-order statements:
\begin{itemize}
\item $\faceq(i,j,n)$ asserts that ${\bf x}[i..i+n-1]={\bf x}[j..j+n-1]$
\item $\novel(i,n)$ asserts that
the first occurrence of the factor
${\bf x}[i..i+n-1]$ is at position $i$ of $\bf x$
\item $\countpp(i,n)$ asserts that
${\bf x}[i..i+n-1]$ is a palindromic periodicity and is the first appearance of this factor in $\bf x$.
\end{itemize}
These can be defined as follows:
\begin{align*}
\faceq(i,j,n) &= \forall u,v \ (u\geq i \andd u<i+n \andd u-i=v-j) \implies
	{\bf x}[u] = {\bf x}[v] \\
 \novel(i,n) &=
\forall j \ (j<i) \implies \neg \faceq(i,j,n)\\
\countpp(i,n) &= \novel(i,n) \andd \ispp(i,n) .
\end{align*}
We  can now form the linear representation for $\pp_{\bf x}(i,n)$
directly from the automaton for
$\countpp$, as described in \cite[Chap.~9]{Shallit:2023}.
\end{proof}

\section{Results for some famous infinite words}

In this section, we use the results of the previous section to prove
a number of results about the palindromic periodicities occurring
in some famous infinite words.

\subsection{The period-doubling word}

Recall that the period-doubling sequence ${\bf pd} = 1011101010111011\cdots$ is the fixed
point of the morphism
$1 \rightarrow 10$, $0 \rightarrow 11$.

\begin{verbatim}
def pdper "At (t>=i & t+p<i+n) => PD[t]=PD[t+p]":
# does PD[i..i+n-1] have period p?
# 15 states
def pdpal "At,u (t>=i & t<i+n & t+u+1=2*i+n) => PD[t]=PD[u]":
# is PD[i..i+n-1] a palindrome?
# 6 states
def pd2pal "Em m<=n & $pdpal(i,m) & $pdpal(i+m,n-m)":
# is PD[i..i+n-1] the concatenation of two palindromes
# 27 states
def pdpp "Ep p>=1 & p<=n & n>=1 & $pdper(i,n,p) & $pd2pal(i,p)":
# is PD[i..i+n-1] a palindromic periodicity?
# 8 states
def pdfaceq "Au,v (u>=i & u<i+n & v+i=u+j) => PD[u]=PD[v]":
# is PD[i..i+n-1]=PD[j..j+n-1]?
# 7 states
def pdnovel "Aj (j<i) => ~$pdfaceq(i,j,n)":
# is PD[i..i+n-1] novel, that is, is it the first occurrence of this factor?
# 5 states
def countpdpp n "$pdpp(i,n) & $pdnovel(i,n)":
# is PD[i..i+n-1] novel and a palindromic periodicity?
# 10 states
\end{verbatim}

Figure~\ref{pdpp-fig} gives the automaton for {\tt pdpp} and
Figure~\ref{countpdpp-fig} gives the automaton for {\tt countpdpp}.
\begin{figure}[htb]
\begin{center}
\includegraphics[width=5.5in]{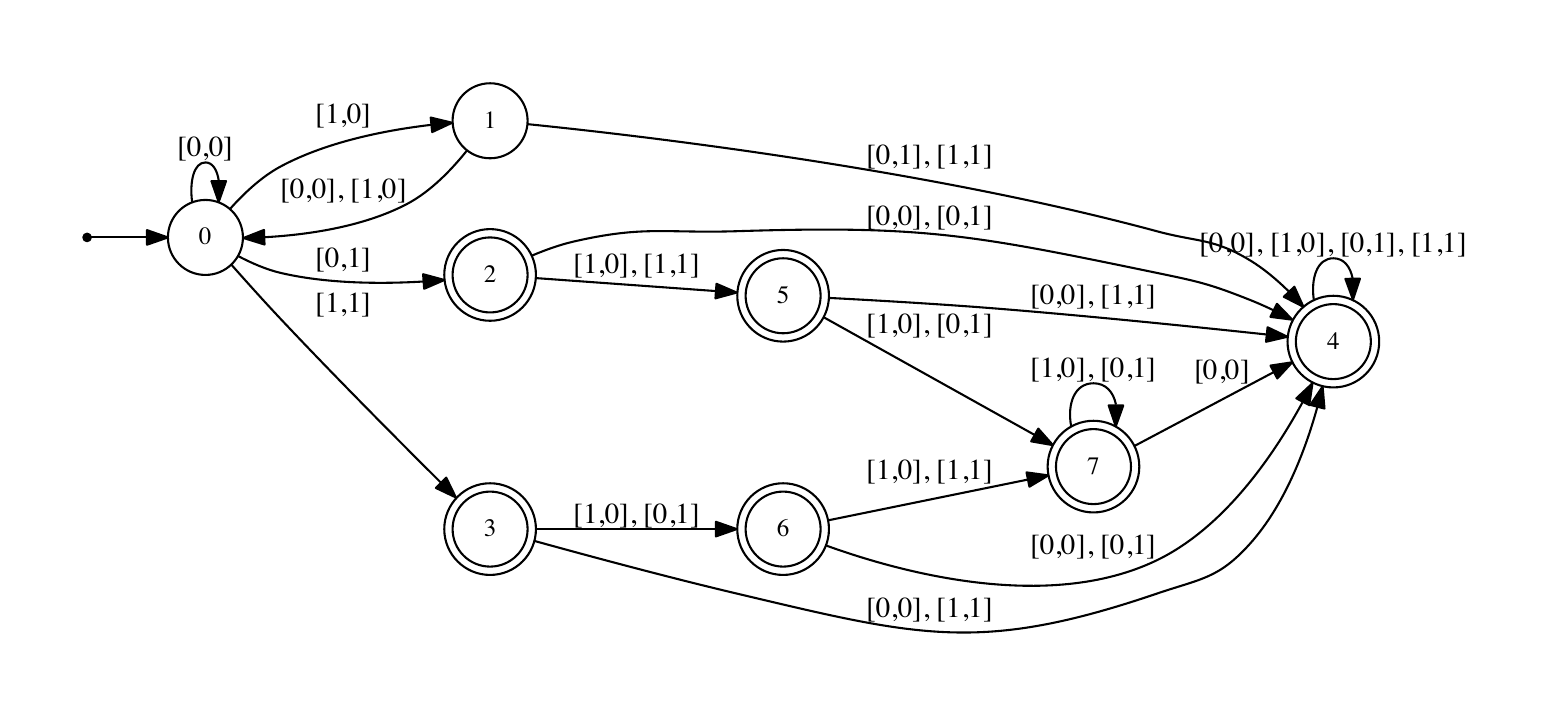}
\end{center}
\caption{Synchronized automaton for {\tt pdpp}; accepts the base-$2$
representation of $i$ and $n$ if ${\bf pd}[i..i+n-1]$ is a palindromic
periodicity.}
\label{pdpp-fig}
\end{figure}

\begin{figure}[htb]
\begin{center}
\includegraphics[width=5.5in]{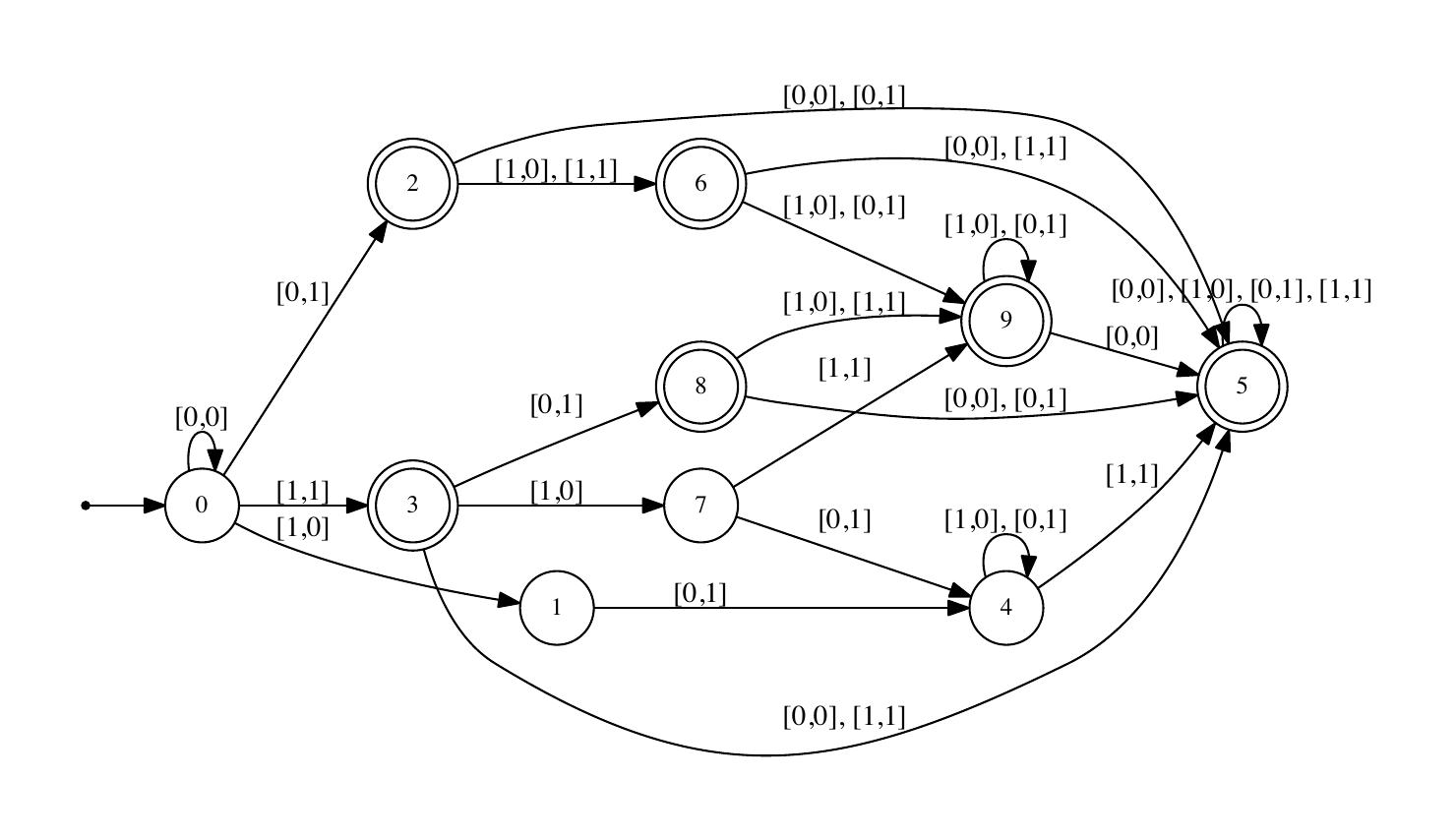}
\end{center}
\caption{Synchronized automaton for {\tt countpdpp}; accepts the base-$2$
representation of $i$ and $n$ if ${\bf pd}[i..i+n-1]$ is novel and
also a palindromic periodicity.}
\label{countpdpp-fig}
\end{figure}

Notice that {\tt pdpp} accepts the
base-$2$ representation of $(12,9)$, since ${\bf pd}[12..20] = 101110111$
is a palindromic periodicity (take $p = 1011101$, $s=11$).
but does not accept $(7,9)$, since ${\bf pd}[7..15] = 010111011$ is not
a palindromic periodicity.  On the other hand,
{\tt countpdpp} does not accept $(12,9)$, since the factor
${\bf pd}[12..20]$ is not novel; it occurs earlier
at ${\bf pd}[8..16]$.

\begin{theorem}
We have $\pp_{\bf pd} (n) \leq 5n/3$ for $n \geq 2$ and
$\pp_{\bf pd} (n) \geq (6n+6)/5$ for $n \geq 3$, and both
bounds are tight, in the sense that they hold with equality infinitely
often.
\end{theorem}

\begin{proof}
First we use {\tt countpdpp} to obtain a linear representation
$(v, \mu, w)$ for $\pp_{\bf pd} (n)$.  It is
$$
v = \left[\begin{smallarray}{c}
1\\
 1\\
 0\\
 0\\
 0\\
 0\\
 0\\
 0\\
 0\\
 0
\end{smallarray} \right];
\quad
\mu(0) = \left[\begin{smallarray}{cccccccccc}
1&1&0&0&0&0&0&0&0&0\\
 0&0&0&0&0&0&0&0&0&0\\
 0&0&0&0&0&1&1&0&0&0\\
 0&0&0&0&0&1&0&1&0&0\\
 0&0&0&0&1&0&0&0&0&0\\
 0&0&0&0&0&2&0&0&0&0\\
 0&0&0&0&0&1&0&0&0&1\\
 0&0&0&0&0&0&0&0&0&0\\
 0&0&0&0&0&1&0&0&0&1\\
 0&0&0&0&0&1&0&0&0&1
\end{smallarray} \right]; \quad
\mu(1) = \left[\begin{smallarray}{cccccccccc}
0&0&1&1&0&0&0&0&0&0\\
 0&0&0&0&1&0&0&0&0&0\\
 0&0&0&0&0&1&1&0&0&0\\
 0&0&0&0&0&1&0&0&1&0\\
 0&0&0&0&1&1&0&0&0&0\\
 0&0&0&0&0&2&0&0&0&0\\
 0&0&0&0&0&1&0&0&0&1\\
 0&0&0&0&1&0&0&0&0&1\\
 0&0&0&0&0&1&0&0&0&1\\
 0&0&0&0&0&0&0&0&0&1
\end{smallarray} \right]; 
\quad
w = \left[\begin{smallarray}{c}
0\\
 0\\
 1\\
 1\\
 0\\
 1\\
 1\\
 0\\
 1\\
 1\
\end{smallarray}
\right].
$$
\end{proof}

We can now use this linear representation to prove the following.
\begin{theorem}
Define 
$ f(0) = 0$ and $f(1) = 2$.  If $n \geq 2$, write
$n = 2^k + r$ for $0 \leq r < 2^k$ and define
$$ f(n) = 
\begin{cases}
	3\cdot 2^{k-1}, &  \text{if $0 \leq r < 2^{k-2}$;} \\
	2^{k+1}, & \text{if $2^{k-2} \leq r < 2^{k-1}$; } \\
	6 \cdot2^{k-1} - r, &  \text{if $2^{k-1} \leq r < 3\cdot 2^{k-2}$;} \\
	7 \cdot2^{k-1} -r, & \text{if $3\cdot2^{k-2} \leq r < 2^k$.}
	\end{cases}
$$
Then $f(n) = \pp_{\bf pd} (n)$ for all $n\geq 0$.
\end{theorem}

\begin{proof}
We create a {\tt Walnut} formula {\tt guess\_pdpp}
of two arguments, $n$ and $z$, and
accepts if and only if $z = f(n)$, based on the formula in the statement
of the theorem.  Then we create a linear representation
computing $f$ with {\tt count2}.
\begin{verbatim}
reg power2 msd_2 "0*10*":
def guess_pdpp "(n=0&z=0)|(n=1&z=2)|
(Ex,r $power2(x) & x<=n & n<2*x & r+x=n &
((4*r<x&2*z=3*x)|(x<=4*r&2*r<x&z=2*x)|(x<=2*r&4*r<3*x&z+r=3*x)|
(3*x<=4*r&r<x&2*z+2*r=7*x)))":
def count2 n "Ez $guess_pdpp(n,z) & i<z"::
\end{verbatim}
Finally, we test equality of the functions computed by {\tt countpdpp} and
{\tt count2} using the method described in \cite{Shallit:2023}.
They are equal.
\end{proof}

\begin{corollary}
We have $\pp_{\bf pd} (n) \leq 5n/3$ for $n \geq 2$ and
$\pp_{\bf pd} (n) \geq (6n+6)/5$ for $n \geq 3$, and both
bounds are tight, in the sense that they hold with equality infinitely
often.
\end{corollary}

\begin{proof}
We can check this with the following {\tt Walnut} code:
\begin{verbatim}
eval lowbnd "An,z (n>=2 & $guess_pdpp(n,z)) => 3*z<=5*n":
eval upbnd "An,z (n>=3 & $guess_pdpp(n,z)) => 5*z>=6*n+6":
\end{verbatim}
and both return {\tt TRUE}.  

Furthermore, it is easy to deduce from similar formulas and examining
the resulting automata that $\pp_{\bf pd} (n) = 5n/3$ if and only
if $n = 3 \cdot 2^k$ for $k \geq 0$, and $\pp_{\bf pd}(n) = (6n+6)/5$ if
and only if $n = 5\cdot 2^k - 1$ for $k \geq 0$.
\end{proof}

In the sections that follow, we report on the same kinds of computations, but
leave many of the details to the reader.

\subsection{The Thue--Morse word}

Recall that the Thue--Morse word
${\bf t} = 0110100110010110 \cdots$
is the fixed point starting with $0$ of the morphism
$0 \rightarrow 01$, $1 \rightarrow 10$.

We use the following {\tt Walnut} code:
\begin{verbatim}
def tmper "At (t>=i & t+p<i+n) => T[t]=T[t+p]":
# 26 states
def tmpal "At,u (t>=i & t<i+n & t+u+1=2*i+n) => T[t]=T[u]":
# 15 states
def tm2pal "Em m<=n & $tmpal(i,m) & $tmpal(i+m,n-m)":
# 61 states
def tmpp "Ep p>=1 & p<=n & n>=1 & $tmper(i,n,p) & $tm2pal(i,p)":
# 28 states
def tmfaceq "Au,v (u>=i & u<i+n & v+i=u+j) => T[u]=T[v]":
# 14 states
def tmnovel "Aj (j<i) => ~$tmfaceq(i,j,n)":
# 8 states
def counttmpp n "$tmpp(i,n) & $tmnovel(i,n)":
# 45 states
\end{verbatim}
This gives us a linear representation 
of rank $45$, too big to present here.
Using this linear representation, we can prove an exact formula
for $\pp_{\bf t}(n)$, as follows:
\begin{theorem}
Let $n \geq 3$ and write $n = 2^k + r$, where $0 \leq r < 2^k$.
If $k$ is even, then 
$$ \pp_{\bf t}(n) = \begin{cases}
	2^{k+1} + 2-2r, & \text{if $0 \leq r < 2^{k-2}$;} \\
	3\cdot 2^k + 2-2r, & \text{if $2^{k-2} \leq r \leq 2^{k-1}$;} \\
	2^{k+2} + 4-4r, & \text{if $2^{k-1} < r < 3\cdot 2^{k-2}$;} \\
	10\cdot2^{k-1} + 4-4r, & \text{if $3\cdot 2^{k-2} \leq r < 2^k$}.
	\end{cases} $$
If $k$ is odd, then
$$ \pp_{\bf t}(n) = \begin{cases}
        3\cdot 2^{k-1} + 2-2r, & \text{if $0 \leq r < 2^{k-3}$;} \\
        2^{k+1} + 2-2r, & \text{if $2^{k-3} \leq r < 2^{k-1}$;} \\
        2^{k+2}-2 , & \text{if $r = 2^{k-1}$}; \\
        3 \cdot 2^{k+1} + 4-4r,& \text{if $2^{k-1} \leq r < 2^k$}.
        \end{cases} $$
\end{theorem}

\begin{proof}
We can carry out the proof using exactly the same ideas as for the period-doubling proof.  The details are omitted.
\end{proof}

\begin{corollary}
We have $\pp_{\bf t} (n) \geq (n+17)/2$ for $n \geq 12$ and
$\pp_{\bf t} (n) \leq (8n-6)/3$ for $n \geq 6$.   Furthermore, these
bounds are sharp, in the sense that the bounds are achieved for
infinitely many $n$.   The lower bound is achieved for $n = 2\cdot4^k - 1$, $k \geq 1$,
and the upper bound is achieved for 
$n = 3 \cdot 4^k$, $k \geq 0$.
\end{corollary}

\subsection{The Rudin--Shapiro word}

Recall that the Rudin--Shapiro word \cite{Rudin:1959,Shapiro:1952} 
counts the number of $11$'s occurring in the base-$2$ representation of $n$,
taken modulo $2$.

\begin{theorem}
For $n \geq 25$ there are no length-$n$ factors of the Rudin--Shapiro word that are palindromic periodicities.  The bound $25$ is optimal, as witnessed by the length-$24$ factor $011110110111100010000100$.
\end{theorem}

\begin{proof}
Easily proved with {\tt Walnut}, as follows:
\begin{verbatim}
def rsper "At (t>=i & t+p<i+n) => RS[t]=RS[t+p]":
def rspal "At,u (t>=i & t<i+n & t+u+1=2*i+n) => RS[t]=RS[u]":
def rs2pal "Em m<=n & $rspal(i,m) & $rspal(i+m,n-m)":
def rspp "Ep p>=1 & p<=n & n>=1 & $rsper(i,n,p) & $rs2pal(i,p)":
eval no25 "~Ei,n $rspp(i,n) & n>=25":
\end{verbatim}
\end{proof}

\subsection{The regular paperfolding word}

Recall that the regular paperfolding word
is defined by the limit of the sequence of words $p_0 = 0$ and
$p_{n+1} = p_n \, 0 \, \overline{p_n}^R$, where $\overline{w}$ denotes the binary complement of $w$.   See \cite{Davis&Knuth:1970}.

\begin{theorem}
For $n \geq 22$ there are no length-$n$ factors of the regular paperfolding word that are palindromic periodicities.  The bound $22$ is optimal, as witnessed by the length-$21$ factor $011000110111001001110$.
\end{theorem}

\begin{proof}
    Easily proved just as for the Rudin--Shapiro word.
\end{proof}

\subsection{The Tribonacci word}
\label{trib-sec}

Recall that the Tribonacci word is the fixed point of the morphism
$0 \rightarrow 01$, $1 \rightarrow 02$, $2 \rightarrow 0$.  See, e.g., \cite{Barcucci&Belanger&Brlek:2004}.

We use the following {\tt Walnut} code:
\begin{verbatim}
def tribper "?msd_trib At (t>=i & t+p<i+n) => TR[t]=TR[t+p]":
# 403 states
def tribpal "?msd_trib At,u (t>=i & t<i+n & t+u+1=2*i+n) => TR[t]=TR[u]":
# 78 states
def trib2pal "?msd_trib Em m<=n & $tribpal(i,m) & $tribpal(i+m,n-m)":
# 181 states
def tribpp "?msd_trib Ep p>=1 & p<=n & n>=1 & $tribper(i,n,p) & $trib2pal(i,p)":
# 40 states
def tribfaceq "?msd_trib Au,v (u>=i & u<i+n & v+i=u+j) => TR[u]=TR[v]":
# 26 states
def tribnovel "?msd_trib Aj (j<i) => ~$tribfaceq(i,j,n)":
# 22 states
def counttribpp n "?msd_trib $tribpp(i,n) & $tribnovel(i,n)":  
# 54 states
\end{verbatim}

Not all factors of $\bf tr$ are palindromic periodicities; e.g., $102$, but
many of them are.   For example, we can verify
that all prefixes are (see Theorem 4 above):
\begin{verbatim}
eval allprefixtrib "?msd_trib An (n>=1) => $tribpp(0,n)":
\end{verbatim}
and {\tt Walnut} returns {\tt TRUE}.

We now compute an explicit formula for the number of
length-$n$ factors of $\bf tr$ that are palindromic periodicities.
Define, as usual, the Tribonacci numbers
by $T_0 = 0$, $T_1 = 1$, $T_2 = 2$ and
$T_n = T_{n-1}+T_{n-2} + T_{n-3}$ for
$n \geq 3$.  We also set $T_i = 0$ for $i < 0$.

\begin{theorem}
Let $n \geq 0$ and $T_k \leq n < T_{k+1}$.   Then
$$
\pp_{\bf tr}(n) = 
\begin{cases}
    2n+1, & \text{if $n\leq (T_{k+1}-T_{k-1}-1)/2$;} \\
    2T_{k+1} + 2T_{k-1} - (2n+1), & \text{if $(T_{k+1}-T_{k-1}-1)/2 < n<T_k + T_{k-1}$}; \\
    T_{k+1} + T_{k-1}, & \text{otherwise.}
\end{cases}
$$
\label{tribb}
\end{theorem}

\begin{corollary}
We have $\pp_{\bf tr} (n) \leq 2n+1$ and
$\pp_{\bf tr} (n) > \alpha n$, where
$\alpha \doteq 1.0873780253841527$ is the
real zero of $X^3 + 2X^2 + 4X-8$.
\end{corollary}
\begin{proof}
For the lower bound, we see from 
Theorem~\ref{tribb} that the local minima
of $\pp_{\bf tr}(n)/n$ occur 
when $n = 2T_{k+1}-2T_k + 1$
and $\pp_{\bf tr}(n) = T_k + T_{k-1} - 1$.
The result now follows from the explicit
formula for $T_k$.   We omit the details.
\end{proof}

\section{Words with few palindromic periodicities}

When we count palindromic periodicities in this section, we
do not count the empty word.

As we have seen above, both the Rudin--Shapiro word and the regular paperfolding word have only finitely many palindromic periodicities as factors.   Indeed, using the linear representation for counting the number of palindromic periodicities of length $n$, we can show that the Rudin--Shapiro word has exactly $334$ palindromic periodicities, and the paperfolding word has exactly $255$ palindromic periodicities.   This suggests the question of finding words with the minimum possible number of palindromic periodicities.

We say a word is {\it aperiodic\/} if it is not ultimately periodic.

\begin{theorem}
Let $\Sigma = \{0,1,\ldots, k-1 \}$ be an alphabet of size
at least $3$.   
\begin{itemize}
    \item[(a)] No infinite word over $\Sigma$ has fewer than $6$ distinct palindromic periodicities.
    \item[(b)] The bound $6$ is optimal because $(012)^\omega$ has $6$.
    \item[(c)] No aperiodic infinite word over $\Sigma$ has $\leq 8$ palindromic periodicities.
    \item[(d)] The bound $8$ is optimal, because the image of $\bf f$ under the map $\tau: 0\rightarrow 0$, $1 \rightarrow 12$ has $9$ palindromic periodicities, namely
    $\{0,1,2,01,12,20,00,001,200\}$.
\end{itemize}
\end{theorem}

The word $\tau({\bf f})$ was previously studied by
the first author  and Luca Zamboni \cite{Fici&Zamboni:2013}.

\begin{proof}
\leavevmode
\begin{itemize}
    \item[(a)]  Without loss of generality we may assume that the first occurrence of a letter $i$ precedes the first occurrence of $j$ for all $i<j$.  Then breadth-first search, where we allow the alphabet to grow with the size of the word, demonstrates that the longest such word with $\leq 5$ palindromic periodicities is of length $5$, namely $00000$.

    \item[(b)] It is easy to see the only palindromic periodicities of $(012)^\omega$ are
    $0, 1, 2, 01, 12, 20$.

    \item[(c)] Again, without loss of generality, we assume that the first occurrence of a letter $i$ precedes the first occurrence of $j$ for all $i<j$.   With this assumption, we claim that
    if $w$ is of length at least $9$ and has $\leq 8$ palindromic periodicities, then
    $w$ is of one of the following forms:
    \begin{center}
    $0 (012)^i \{ \epsilon, 0, 01 \}$,\\
    $(012)^i \{ \epsilon, 0, 2, 3, 00, 01, 03, 011, 013 \}$,\\
    $(0123)^i \{ \epsilon, 0, 01, 012 \}$,\\
    $0 (123)^i \{ \epsilon, 1, 12 \}$.
    \end{center}

 The claim can be easily verified for $9 \leq |w| \leq 12$.   It is easy to check that all words with the stated property of length $12$ can be written uniquely in the form $x y^i z$ with $i \geq 2$
    $x \in \{ \epsilon, 0 \}$, $y \in \{ 012, 123,0123 \}$,
    and $|z| \leq 3$.   Let $w$ be a shortest counterexample
    to the claim with $|w| > 12$.  Let $x y^i z$ be the unique
    factorization of the prefix of $w$ of length $12$; then
    $i \geq 2$.   Write $w = x y^j z'$ with $j$ as large
    as possible.   It is now easy to check, by considering
    the possible suffixes $y^2 z'$ of $w$, that either $z'$
    begins with $y$ (so $j$ was not maximal, a contradiction),
    or $z'$ is one of the words given in the characterization,
    a contradiction.

    \item[(d)]  We use {\tt Walnut}.   Here it is possible to compute an automaton for $\tau({\bf f})$ because the
    number of occurrences of $0$ (resp., $1$) in a prefix of length $n$
    of $\bf f$ is synchronized; see \cite[Sect.~10.11]{Shallit:2023}.   The following
    {\tt Walnut} code checks that $\tau({\bf f})$ has no palindromic periodicities of length $>3$; it is then 
    easy to enumerate them by hand.   The code for the first five automata is taken from   \cite[Sect.~10.11]{Shallit:2023}.
\begin{verbatim}
reg shift {0,1} {0,1} "([0,0]|[0,1][1,1]*[1,0])*":
def phin "?msd_fib (s=0 & n=0) | Ex $shift(n-1,x) & s=x+1":
def noverphi "?msd_fib Et $phin(n,t) & s+n=t":
def fibpref0 "?msd_fib $noverphi(n+1,s)":
def fibpref1 "?msd_fib Eu $fibpref0(n,u) & n=s+u":

def img "?msd_fib Ew,x,y,z $fibpref0(q,x) & $fibpref1(q,y) &
   $fibpref0(q+1,w) & $fibpref1(q+1,z) & x+2*y<=n & w+2*z>n &
   r+x+2*y=n":
# img(n) = (q,r) means the n'th position of the image 
# tau(f) is the r'th letter of phi(f[q])

def img0 "?msd_fib Eq,r $img(n,q,r) & F[q]=@0":
def img1 "?msd_fib Eq,r $img(n,q,r) & F[q]=@1 & r=0":
def img2 "?msd_fib Eq,r $img(n,q,r) & F[q]=@1 & r=1":
combine NF img0=0 img1=1 img2=2:

def nfper "?msd_fib At (t>=i & t+p<i+n) => NF[t]=NF[t+p]":
def nfpal "?msd_fib At,u (t>=i & t<i+n & t+u+1=2*i+n) => NF[t]=NF[u]":
def nf2pal "?msd_fib Em m<=n & $nfpal(i,m) & $nfpal(i+m,n-m)":
def nfpp "?msd_fib Ep p>=1 & p<=n & n>=1 & $nfper(i,n,p) & $nf2pal(i,p)":
eval nfpp3 "?msd_fib Ai,n $nfpp(i,n) => n<=3":
\end{verbatim}

and {\tt Walnut} returns {\tt TRUE} for the last assertion.
To check that $\tau({\bf f})$ is aperiodic, we check that it
has no $4$th powers:
\begin{verbatim}
eval no4 "?msd_fib ~Ei,n,p n>=1 & p>=1 & $nfper(i,n,p) & n>=4*p":
\end{verbatim}
\end{itemize}
\end{proof}

The binary case is similar but a bit more complicated.

\begin{theorem}
Let $\Sigma = \{0,1 \}$.
\begin{itemize}
    \item[(a)] No infinite word over $\Sigma$ has fewer than $30$ distinct palindromic periodicities.
    \item[(b)] The bound $30$ is optimal because $(001011)^\omega$ has $30$.
    \item[(c)] No aperiodic infinite binary word has $\leq 43$ palindromic periodicities.
    \item[(d)] The bound $43$ is optimal, because the image of $\bf f$ under the map $\varphi: 0\rightarrow 0$, $1 \rightarrow 01101$ has $44$ palindromic periodicities.
\end{itemize}
\end{theorem}

The word $\varphi({\bf f})$ was previously studied by
the first author and L.~Zamboni \cite{Fici&Zamboni:2013}.

\begin{proof}[Proof sketch.]
\leavevmode
\begin{itemize}
    \item[(a)] Breadth-first search shows that the longest binary words with $\leq 29$ palindromic periodicities are of length $29$; namely, $0^{29}$ and $1^{29}$.

    \item[(b)] It is easy to check that the only palindromic periodicities in $(001011)^\omega$ are
    \begin{align*}
        &0, 1, 00, 01, 10, 11, 001, 010, 011, 100, 101, 110, 0010, 0101, 0110, 1001, 1011, \\
        &1100, 00101, 01011, 01100,
       10010, 10110, 11001, 010110, 011001, 100101,\\
       &0110010, 1011001, 10010110.
    \end{align*}

    \item[(c)]   We claim that every sufficient large finite binary word having $\leq 43$ palindromic periodicities is of the form $x y^i z$ where $|x|, |z| \leq 5$ and $y$ is a conjugate of one of members of the set $B$,
    where 
    \begin{align*}
     B &= \{ 001011, 001101, 0001011, 0001101, 0010111, 0011101, 00001011, 00001101, \\
   & \quad \quad 00010111,  00011101, 00101011, 00101111, 00110101, 00111101\}.    
    \end{align*}
     The argument is similar as for the previous theorem.

    \item[(d)] We use {\tt Walnut}. Again it is possible to compute an automaton for $\varphi( {\bf f} )$.
The following {\tt Walnut} code checks that $\varphi({\bf f})$ has no palindromic periodicities of
length $> 9$; it is then easy to enumerate them by hand.
\begin{verbatim}
def img2 "?msd_fib Ew,x,y,z $fibpref0(q,x) & $fibpref1(q,y) & 
   $fibpref0(q+1,w) & $fibpref1(q+1,z) & x+5*y<=n & w+5*z>n & r+x+5*y=n":
# img2(n) = (q,r) means the n'th position of the image 
# phi(f) is the r'th letter of phi(f[q])

def img20 "?msd_fib Eq,r $img2(n,q,r) & ((F[q]=@0)|(F[q]=@1 & (r=0|r=3)))":
def img21 "?msd_fib Eq,r $img2(n,q,r) & F[q]=@1 & (r=1|r=2|r=4)":
combine QF img20=0 img21=1:
# 61 states

def qfper "?msd_fib At (t>=i & t+p<i+n) => QF[t]=QF[t+p]":
#17157 states
def qfpal "?msd_fib At,u (t>=i & t<i+n & t+u+1=2*i+n) => QF[t]=QF[u]":
#119 states, largest intermediate automaton was 58110620 states!
def qf2pal "?msd_fib Em m<=n & $qfpal(i,m) & $qfpal(i+m,n-m)":
#187 states
def qfpp "?msd_fib Ep p>=1 & p<=n & n>=1 & $qfper(i,n,p) & $qf2pal(i,p)":
#203 states

eval nqpp9 "?msd_fib Ai,t $qfpp(i,t) => t<=9":
\end{verbatim}

We can also check that $\varphi({\bf f})$ is aperiodic.
\begin{verbatim}
eval no42 "?msd_fib ~Ei,n,p n>=1 & p>=1 & $qfper(i,n,p) & n>=4*p":
\end{verbatim}
    
\end{itemize}
\end{proof}

\section{For further research}

Asymptotically, how many palindromic periodicities of length $n$ are there for binary words? 

The first few values of the sequence are: 
\begin{align*}
&2, 4, 8, 16, 32, 58, 108, 190, 336, 560, 948, 1574, 2568,\\
&4116, 6596, 10444, 16320, 25488, 39216, 60690, 92204.
\end{align*}
It is sequence~\seqnum{A374495} in the {\it On-Line Encyclopedia of Integer Sequences} \cite{Sloane:2024}.  More terms there have now been computed by  Michael S. Branicky.

Obviously, the growth rate of the language of palindromic periodicities is at least $\Omega(2^{n/2})$, since palindromes are particular cases of palindromic periodicities.


\end{document}